\documentclass{article}
\usepackage{fullpage}

\usepackage[a4paper, total={5.5in, 8.5in}]{geometry}

\usepackage{amsmath,amsthm,amssymb,mathtools}
\usepackage{graphicx}
\usepackage[usenames,dvipsnames]{xcolor}

\graphicspath{ {images/} }
\usepackage{hyperref} 

\newcommand{\atw}{\operatorname{Aut}}
\newcommand{\aut}{\operatorname{Aut}}

\newtheorem{theorem}{Theorem}

\newtheorem{lemma}[theorem]{Lemma}
\newtheorem{corollary}[theorem]{Corollary}
\newtheorem{proposition}[theorem]{Proposition}

\newtheorem{example}{Example}

\theoremstyle{definition}

\makeatletter
\def\moverlay{\mathpalette\mov@rlay}
\def\mov@rlay#1#2{\leavevmode\vtop{%
   \baselineskip\z@skip \lineskiplimit-\maxdimen
   \ialign{\hfil$\m@th#1##$\hfil\cr#2\crcr}}}
\newcommand{\charfusion}[3][\mathord]{
    #1{\ifx#1\mathop\vphantom{#2}\fi
        \mathpalette\mov@rlay{#2\cr#3}
      }
    \ifx#1\mathop\expandafter\displaylimits\fi}
\makeatother

\newcommand{\bigcupdot}{\charfusion[\mathop]{\bigcup}{\cdot}}

\newcommand{\LabelQuote}[2]{\vspace{0.5cm}%
     \parbox{13cm}{\em #1}\hspace*{0.5cm}(#2)\\[0.5cm]}

\begin{document}
\title{On the Automorphisms of Token Graphs Generated by $2$-cuts with the Same Neighbours\footnote{Supported by CONACYT FORDECYT-PRONACES/39570/2020}}
\author{Ruy Fabila-Monroy \thanks{Departamento de Matem\'aticas, CINVESTAV.} 
\footnote{\tt{ruyfabila@math.cinvestav.edu.mx}}    
\and Sergio Gerardo G\'omez-Galicia\footnotemark[2] \footnote{\texttt{sgomez@math.cinvestav.mx}} \and Daniel Gregorio-Longino\footnotemark[2] \footnote{\texttt{dgregorio@math.cinvestav.mx }}\and Teresa I. Hoekstra-Mendoza\thanks{Centro de Investigaci\'on en Matem\'aticas, Guanajuato, Mexico. \texttt{maria.idskjen@cimat.mx}} \and 
Ana  Trujillo-Negrete\thanks{Centro de Modelamiento Matem\'atico (CNRS IRL2807), Universidad de Chile, Santiago, Chile.
	Partially supported by ANID/Fondecyt Postdoctorado 3220838, and by ANID Basal Grant CMM
	FB210005. \texttt{ltrujillo@dim.uchile.cl}}}
\maketitle

\begin{abstract}
Let $G$ be a connected graph on $n$ vertices and $1 \le k \le n-1$ an integer. The $k$-token graph of $G$ is the graph $F_k(G)$ whose vertices are all the $k$-subsets of vertices of $G$, two of which are adjacent whenever their symmetric difference is an edge of $G$. Every automorphism of $G$ induces an automorphism of $F_k(G)$ in a natural way. Suppose that $S:=\{x,y\}$ is a cut set of $G$, such that $x$ and $y$ 
have the same neighbours in $G\setminus \{x,y\}$. In this paper we show that there exist a large number
of automorphisms of $F_k(G)$ defined by $S$ that are not induced by automorphisms of $G$. We also describe
the group produced by all such $2$-cuts of $G$. 
\end{abstract}

\section*{Introduction}

Let $G$ be a graph on $n$ vertices, and $1 \le k \le n-1$ an integer. 
The \emph{$k$-token graph} of $G$ is the graph  $F_k(G)$, whose vertices are all $k$-subsets of vertices of $G$; in $F_k(G)$ two vertices $A$ and $B$  are adjacent whenever their symmetric difference $A\triangle B$ is an edge of $G$.
Token graphs have been defined independently under different names at least four times~\cite{double_vertex,Johns,rudolph,ktuple, Token}.
In this paper we use the following interpretation given in~\cite{Token}. Place $k$ indistinguishable tokens at the vertices of $G$, with at most
one token per vertex. Construct a new graph whose vertices are all possible token configurations. Two token configurations
are adjacent if one can be obtained from the other by sliding a token along an edge of $G$ to an unoccupied vertex. The graph
so constructed is isomorphic to $F_k(G)$. In Figure~\ref{fig:example}, a graph is shown together with
its $2$-token graph.

\begin{figure}[ht]
	\begin{center}
		\includegraphics[width=.8\textwidth]{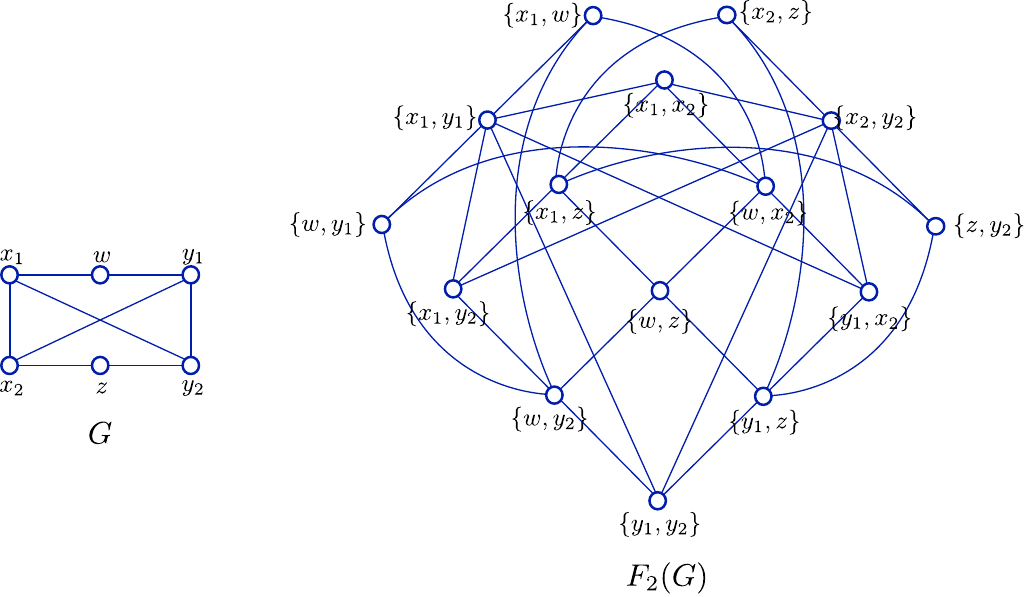}
	\end{center}
	\caption{A graph $G$ and its $2$-token graph}\label{fig:example}
\end{figure}

Let $\aut(G)$ be the group of automorphisms of $G$. Let $\iota$ be the function that maps every $f \in \aut(G)$, to the function
\[\iota(f)(\{a_1,\dots,a_k\})=\{f(a_1),\dots,f(a_k)\}, \] 
for every $\{a_1,\dots,a_k\} \in V(F_k(G))$. It is straightforward to prove that $\iota(f)$ is an automorphism of $F_k(G)$.
In fact, as Ibarra and Rivera~\cite{sofiamanuel} have shown, $\iota$ is an injective group homomorphism. 

Let $\mathfrak{c}$ be the function that maps every $A \in V(F_k(G))$ to
\[\mathfrak{c}(A):=V(G)\setminus A.\]
This function is an isomorphism from $F_k(G)$ to $F_{n-k}(G)$, and thus,
 an automorphism of $F_k(G)$, whenever $k=n/2$. In this case, we call $\mathfrak{c}$ \emph{the complement automorphism}.
Whenever an automorphism $g \in \aut(F_k(G))$ is such that there exists $f \in  \aut(G)$ such
that
\[g =\iota(f) \textrm{ or }  g =\mathfrak{c} \circ\iota(f),\]
 we say that $g$ is an \emph{induced automorphism} of $F_k(G)$. It was shown
in \cite{C4Diamond} that  for all $f \in \aut(G)$ we have that
\[\mathfrak{c} \neq \iota(f).\]
Furthermore, if $k=n/2$, then
\[\mathfrak{c} \circ \iota(f)=\iota(f) \circ \mathfrak{c},\]
for every $f \in \aut(G)$. Therefore,
\[\iota(\aut(G)) \le \aut(F_k(G))\] and
\[\iota(\aut(G))\times \mathbb{Z}_2 \le \aut(F_k(G))  \textrm{ when } k=n/2. \]

Recently, two examples have been found that show that these inclusions can be far from being equalities. 
\begin{itemize}
\item Zhang, Zhou, Lee, Li and Xie \cite{2t_cube} show that if $G$ is the Cartesian product of $r$ prime graphs, then
\[
    \mathbb{Z}_{2}^{r-1}\rtimes\atw(G)\leq \atw(F_{2}(G));
\]
and

\item Fabila-Monroy and Trujillo-Negrete~\cite{autbipartite} show that for $1<k<n+1$
 \[
	\atw(F_{k}(K_{2,n})) =
	\begin{cases}
        \mathbb{Z}_{2}\wr_{\binom{[n]}{k-1}}S_{n} & \text{ if }  k\neq \frac{n+2}{2}, \\
		 \mathbb{Z}_{2}\wr_{\binom{[n]}{k-1}}\left ( S_{n}\times\mathbb{Z}_{2}\right ) & \text{otherwise}.
	\end{cases}
\]
\end{itemize}
Where $\binom{[n]}{k-1}$ is the set of all subsets of $[n]:=\{1,\dots,n\}$ of size $k-1$; and a permutation $\pi \in S_n$
acts on $\binom{[n]}{k-1}$ by sending every element $\{x_1,\dots,x_{k-1}\}$ to $\{\pi(x_1),\dots,\pi(x_{k-1})\}$.

In this paper we generalize the result of~\cite{autbipartite}. For definitions and concepts of group theory we refer the reader
to the Preliminaries section of~\cite{autbipartite}, Rotman's book on Group Theory~\cite{rotman}, and the book
by Godsil and Royle on Algebraic Graph Theory~\cite{godsil}. 
\section*{2-cuts with the same neighbours}

Given graph $H$, let $C(H)$ be the set of connected components of $H$, and $c(H):=|C(H)|$. 
Throughout the rest of the paper, assume that $G$ is connected.
Let $S:=\{x,y\}$ be a cut set of $G$, such that $x$ and $y$ have the same neighbours in $V(G)\setminus S$.
We say that $S$ is a \emph{2-cut with the same neighbours} of $G$. In the graph $G$ of Figure~\ref{fig:example},
$\{x_1,y_1\}$ and $\{x_1,y_2\}$ are $2$-cuts with the same neighbours.

\begin{lemma}\label{lem:2cut}
 If $S:=\{x,y \}$ and $S':=\{w,z\}$ are two distinct 2-cuts with the same neighbours of $G$, 
 then
 \begin{itemize}
 \item[$a)$] $S$ and $S'$ are disjoint; and
 
 \item[$b)$] if $G$ is not a $4$-cycle, then $w$ and $z$ are in the same connected component of $G \setminus S$. 
 \end{itemize}
\end{lemma}
\begin{proof}
 Suppose $|S \cap S'|=1$. Without loss of generality, suppose that $x=z$.
 Suppose that $u$ is a neighbour of $y$ distinct from $x$ and $w$. Thus, $u$ is also a neighbour of $x$, and this  implies that it is also a neighbour of $w$. Suppose now that $u$ is a neighbour of $x$ distinct from $y$ and  $w$. Thus, $u$ is also neighbour of $w$. Therefore, every neighbour of $x$ or $y$  distinct from $x,y$ or $w$ is also a neighbour of $w$; this contradicts the assumption that $S$ is a cut set of $G$. 
 Therefore, $S$ and $S'$ are disjoint. This proves $a)$.
 
 Suppose that $w$ and $z$ lie in distinct connected components,
 $G_1$ and $G_2$ of $G\setminus S$, respectively. Note that there  cannot exists a vertex in $G_1$ adjacent to $w$, as otherwise this vertex would be adjacent to $z$ as well. Similarly, there 
  cannot exists a vertex in $G_2$ adjacent to $z$. Therefore,  $G_1$ and $G_2$
 consist of only $w$ and $z$, respectively. Implying that $\{w,z\}$ is not a cut set of $G$,
 unless $G$ is a $4$-cycle. This proves $b)$. 
\end{proof}

For every $A \in V(F_k(G))$ such that $|A\cap S|=1$, let 
\[A^S= (A\setminus  S)\cup ( S\setminus A).\] 
Thus, $A^S$ is obtained from $A$ by moving the token at the only vertex in $A \cap  S$ to
the vertex of $ S$ not in $A$. Note that $(A^S)^S=A$. The next corollary follows directly from Lemma~\ref{lem:2cut}.
\begin{corollary}\label{cor:2cut}
 Let $S_1$ and $S_2$ be two distinct $2$-cuts with the same neighbours of $G$. Then
 for every $A \in V(F_k(G))$ such that $|A\cap S_1|=|A\cap S_2|=1$, we have that
 \[\left (A^{S_1} \right )^{S_2}=\left (A^{S_2} \right )^{S_1}.\]
\end{corollary}
\qed

\subsection*{Automorphisms generated by a $2$-cut with the same neighbours }

Let $c_{ S}:=|C(G\setminus  S)|$; and let $G_1,\dots,G_{c_{ S}}$ be the connected components
of $G\setminus S$ in some fixed order. Let $n_1,\dots,n_{c_{ S}}$ be the cardinalities of $V(G_1),\dots,V(G_{c_{ S}})$, respectively.
Let \[T_{ S}:=\left \{(k_1,\dots,k_{c_{S}}) : \textrm{ every } k_i \textrm{ is an integer, } 0 \le k_i \le n_i \textrm{ and } \sum_{i=1}^{c_{S}} k_i=k-1 \right \}.\]
In what follows we consider other such sets of tuples for different $2$-cuts with the same neighbours of $G$.
To avoid confusion, we occasionally  write 
$(k_1,\dots,k_{c_S})_{S_1}$ to refer to the element $(k_1,\dots,k_{c_S})$ \emph{in} $T_{S_1}$.

For every $t \in T_{ S}$, let 
\[V_{S,t}:=\left \{A \in V(F_k(G)): |A \cap V(G_i)|=t(i) \textrm{ for every } 1 \le i \le c_{S} \right \}.\]
%
%
%
%
%
Let $\varphi_{S,\alpha}:V(F_k(G)) \to V(F_k(G))$  be the map that sends every $A \in V(F_k(G))$ to
\[
    \varphi_{S,\alpha}(A):=
    \begin{cases}
        A^S & \textrm{ if there exists } t \in \alpha \textrm{ such that } A \in V_{ S,t};\\
        A & \textrm{otherwise.}
    \end{cases}
\]

\begin{example}
 Let $G$ be as in Figure~\ref{fig:example}, and $S=\{x_1,y_1\}$. Note that $G\setminus S$, has two connected components,$G_1$ and $G_2$, whose vertex set are equal to $\{w\}$ and $\{x_2,z,y_2\}$, respectively.
 Let $\alpha=(0,1)$ and  set $k=2$. We have that:
 \[
\begin{array}{rcl@{\quad}rcl@{\quad}rcl}
    \varphi_{S,\alpha}(\{x_1,x_2\}) &=& \{y_1,x_2\}; &
    \varphi_{S,\alpha}(\{x_1,z\}) &=& \{y_1,z\}; &
    \varphi_{S,\alpha}(\{x_1,y_2\}) &=& \{y_1,y_2\}; \\[1em]
    \varphi_{S,\alpha}(\{y_1,x_2\}) &=& \{x_1,x_2\}; &
    \varphi_{S,\alpha}(\{y_1,z\}) &=& \{x_1,z\}; &
    \varphi_{S,\alpha}(\{y_1,y_2\}) &=& \{x_1,y_2\};\\[1em]
    \varphi_{S,\alpha}(\{w,x_2\}) &=& \{w,x_2\}; &
    \varphi_{S,\alpha}(\{w, z\}) &=& \{w, z\}; &
    \varphi_{S,\alpha}(\{w, y_2\}) &=& \{w, y_2\};\\[1em]
    \varphi_{S,\alpha}(\{w,x_1\}) &=& \{w,x_1\}; &
    \varphi_{S,\alpha}(\{x_1,y_1\}) &=& \{x_1, y_1\}; &
    \varphi_{S,\alpha}(\{w,y_1\}) &=& \{w,y_1\};\\[1em]
     \varphi_{S,\alpha}(\{x_2,y_2\}) &=& \{x_2,y_2\}; &
    \varphi_{S,\alpha}(\{x_2,z\}) &=& \{x_2, z\}; &
    \varphi_{S,\alpha}(\{z,y_2\}) &=& \{z,y_2\}.\\[1em]
\end{array}
\]
\qed
\end{example}

\begin{proposition}\label{prop:varphi}
 $\varphi_{S,\alpha}$ is an automorphism of $F_k(G)$.
 Moreover, $\varphi_{S,\alpha}$ is an induced automorphism of $F_k(G)$ if and only 
 if $\alpha = \emptyset$ or $\alpha = T_{ S}$. 
\end{proposition}
\begin{proof}
 It is straightforward to show that $\varphi_{S,\alpha}$ is bijective. 
 Let $AB$ be an edge of $F_k(G)$, such that $B$ is obtained
 from $A$ by sliding a token along an edge $ab$ of $G$, with $a \in A$ and $b \in B$.
 
 Suppose that $a,b\notin  S$. Thus, $a$ and $b$ are in the same connected component of $G \setminus S$.
 This implies that there exists $t \in T_S$ such that $A,B \in V_{S,t}$. Since $A \triangle B= \{a,b\}$,
 we have that $\varphi_{S,\alpha}(A) \cap S =\varphi_{S,\alpha}(B) \cap S$,
 $A\setminus S= \varphi_{S,\alpha}(A) \setminus S$ and $B\setminus S= \varphi_{S,\alpha}(B) \setminus S$.
 Therefore,  $\varphi_{S,\alpha}(A) \triangle \varphi_{S,\alpha}(B)=\{a,b\}$, and $\varphi_{S,\alpha}(A)\varphi_{S,\alpha}(B)$ is an edge of $F_k(G)$.

 Suppose that $\{a,b\}= S$. Thus, there exists $t \in T_S$ such that $A,B \in V_{S,t}$.
 If $t \in \alpha$, then  $\varphi_{S,\alpha}(A) \cap S=\{b\}$ and
  $\varphi_{S,\alpha}(B) \cap S=\{a\}$; if $t \notin \alpha$, then  $\varphi_{S,\alpha}(A) \cap S=\{a\}$ and
  $\varphi_{S,\alpha}(B) \cap S=\{b\}$. Since $A\setminus S= \varphi_{S,\alpha}(A) \setminus S$ and $B\setminus S= \varphi_{S,\alpha}(B) \setminus S$, in both cases we have that  $\varphi_{S,\alpha}(A) \triangle \varphi_{S,\alpha}(B)=\{a,b\}$, and $\varphi_{S,\alpha}(A)\varphi_{S,\alpha}(B)$ is an edge of $F_k(G)$.  

Finally, suppose that exactly one of
  $a$ and $b$ is in $ S$. Without loss of generality assume that $a \in  S$ and $b \notin  S$. 
  Let $\overline{a}$
  be the only vertex in $ S\setminus \{a\}$. Since $a$ and $b$ are adjacent in $G$, 
  $\overline{a}$ and $b$ are adjacent. 
  Suppose that $|S \cap A|=1$. Thus, $|S \cap B|=0$. This implies
  that $\varphi_{S,\alpha}(A)=A$ or $\varphi_{S,\alpha}(A)=A \setminus\{a\} \cup \{\overline{a}\}$,
  and $\varphi_{S,\alpha}(B)=B$. In both cases we have that  $\varphi_{S,\alpha}(A)\varphi_{S,\alpha}(B)$ is an edge of $F_k(G)$.  Suppose that $|S \cap A|=2$. Thus, $|S \cap B|=1$. This implies
  that $\varphi_{S,\alpha}(A)=A$,
  and $\varphi_{S,\alpha}(B)=B$ or $\varphi_{S,\alpha}(B)=B \setminus \{\overline{a}\} \cup \{a\}$. In both cases we have that  $\varphi_{S,\alpha}(A)\varphi_{S,\alpha}(B)$ is an edge of $F_k(G)$. 
  
 Thus, in every case we have that $\varphi_{S,\alpha}$ is automorphism of $F_k(G)$.
  If $\alpha=\emptyset$, then $\varphi_{S,\alpha}$ is the identity. If $\alpha=T_{ S}$, then $\varphi_{S,\alpha}$ is induced by the automorphism of $G$ that swaps the two vertices in $S$, while leaving the remaining vertices fixed.

  Suppose that $\alpha \neq \emptyset, T_{ S}$ and let $S=:\{x,y\}$. Suppose for a contradiction that there exists an automorphism $f$ of $G$ such that 
  \[\varphi_{S,\alpha}=\iota(f) \textrm{ or }\varphi_{S,\alpha}=\mathfrak{c} \circ \iota(f).\] 
  In either case let $z_1,\dots,z_m$,  be the vertices in the orbit of $x$ under $f$.
  That is:
  \[z_1=x; z_i=f(z_{i-1}) \textrm{ for all } 2 \le i \le m; \textrm{ and } f(z_m)=x.\]
   Let $t_1 \in \alpha$ and $t_2 \in T_{ S} \setminus \alpha$. Let now $A \in V_{ S,t_1}$, such that
  $x \in A$, and
  let $B \in V_{S,t_2}$, such that $x \in B$.
  
  Suppose that $\varphi_{S,\alpha}=\iota(f)$. 
  Since $\varphi_{S,\alpha}(A)=A \setminus \{x\}\cup \{y\}$, we have that $y$ is in the orbit
  of $x$ under $f$. Since $\varphi_{S,\alpha}\circ \varphi_{S,\alpha}=e$, we have that $m=2$, and
  $f(x)=y$. This is a contradiction to the fact that $x \in B$ and $y \notin B =\varphi_{S,\alpha}(B)$.
  
  Suppose that $\varphi_{S,\alpha}=\mathfrak{c} \circ \iota(f)$. Thus, $k=n/2$. Note that $x \in \iota(f)(A)$ and $u \notin \iota(f)(A)$ for all $u \in A \setminus \{x\}$. Let $A':=A \setminus \{x\}$
  and $\overline{A'}:=V(G)\setminus (A'\cup S)$. If $f(x)=x$, then $x \in \iota(f)(B)$ and
  $x\notin \mathfrak{c} \circ \iota(f)(B)= \varphi_{S,\alpha}(B)=B$---a contradiction. By recalling that 
  $\varphi_{S,\alpha}(A)=\mathfrak{c} \circ \iota(f)(A)=A \setminus \{x\}\cup \{y\}$ and the fact that
  $f(x)\neq x$, we have that:
  \begin{itemize}
   \item[$(a)$] $f(x) \in \overline{A'}$;
   \item[$(b)$] $z_{m} \in A'$; and
   \item[$(c)$] $f(A'\setminus \{z_{m}\})=\overline{A'}\setminus \{f(x)\}.$
  \end{itemize}
  Since $x \notin \iota(f)(B)$ and $f(x) \in \iota(f)(B)$, 
  we have that  $z_{m} \notin B$ and $f(x) \notin   \mathfrak{c}\circ\iota(f)(B)= \varphi_{S,\alpha}(B)=B$.
  Therefore, \[|B \cap ((A'\setminus \{z_{m}\})\cup (\overline{A'} \setminus \{f(x)\}))|=k-1=n/2-1.\]
  
  Since $|A'\setminus \{x_{m}\}|= |\overline{A'} \setminus \{f(x)\}|=n/2-2$, there exists
  a vertex $v \in A'\setminus \{z_{m}\}$, such that $v \in B$ and $f(v) \in B$.
  Therefore, $f(v) \in \iota(f)(B)$; which implies that $f(v) \notin \varphi_{S,\alpha}(B)=B$---a contradiction.
\end{proof}

Since 
\begin{equation} \label{eq:id} \varphi_{S,\alpha}\circ \varphi_{S,\alpha}=e \textrm{(the identity)} , 
 \end{equation}
we have that
 $\varphi_{S,\alpha}^{-1}=\varphi_{S,\alpha}.$
Let $\beta \subseteq T_{ S}$. Note that
$\varphi_{S,\alpha} \circ\varphi_{S,\beta}=\varphi_{S,\alpha \triangle \beta}$.  We can associate to every $\alpha \subseteq T_{ S}$ its characteristic
vector indexed by the elements of $T_{S}$. These observations show that
\begin{equation}\label{eq:2}
N_{S} := \left \langle \{\varphi_{S,\alpha}: \alpha \subseteq T_{ S}\} \right \rangle \cong \mathbb{Z}_2^{|T_S|}.
\end{equation}

\subsection*{Automorphisms generated by all $2$-cuts with the same neighbours }

Let $Q(G)$ be the set of all $2$-cuts with the same neighbours of $G$, and $q(G):=|Q(G)|$. Let $S_1,\dots,S_{q(G)}$ be the
elements of $Q(G)$ in any given but fixed order. For every $1 \le i \le q(G)$, let 
\[\{x_i,y_i\}:=S_i.\]
For every $1 \le i \le q(G)$ assume that the connected components of $G \setminus S_i$ are sorted
in some fixed order.
Let $\mathcal{T}_G$ be the disjoint union of the $T_{S_i}$, that is
\[\mathcal{T}_G:=\bigcupdot_{i=1}^{q(G)} T_{S_i}.\]
Let $\beta \subseteq \mathcal{T}_G$, and for every $1 \le i \le q(G)$ let 
\[\beta_i:=\beta \cap T_{S_i};\]
thus, $\beta_i \subseteq T_{S_i}$.
Let $\psi_{\beta}:V(F_k(G)) \to V(F_k(G))$ be the map that sends every $A \in V(F_k(G))$ to
\[\psi_{\beta}(A):=\varphi_{S_1,\beta_1} \circ \varphi_{S_2,\beta_2} \circ \cdots \circ \varphi_{S_{q(G)},\beta_{q(G)}}(A).\]

From Proposition~\ref{prop:varphi}, we directly obtain the following.
\begin{corollary}\label{cor:psi}
 For every $\beta \subseteq \mathcal{T}_G$, we have that $\psi_{\beta}$ is 
 an automorphism of $F_k(G)$.
 Moreover, $\psi_{B}$ is an induced automorphism of $F_k(G)$ if and only 
 if $\beta_i = \emptyset$ for all $1 \le i \le q(G)$, 
 or $\beta_i = T_{S_i}$ for all $1 \le i \le q(G)$. 
\end{corollary}
\qed

\begin{example}
 Let $G$ be as in Figure~\ref{fig:example}, $S_1:=\{x_1,y_1\}$ and $S_2:=\{x_2,y_2\}$. Let $\{w\}$ and 
 $\{x_2,z,y_2\}$ be the vertex sets of the connected components of $G\setminus S_1$, respectively (in order); 
 and  let $\{z\}$ and 
 $\{x_1,w,y_1\}$ be the vertex sets of the connected components of $G\setminus S_2$, respectively (in order).
 Let $t_1=(0,1)_{S_1}, t_2=(0,1)_{S_2}$ and $\beta=\{t_1,t_2\}$; thus,
 $\beta_1=\{t_1\}$ and $\beta_2=\{t_2\}$. We have that:

 \[
 \begin{array}{lclclcl}
    \psi_\beta (\{x_1,x_2\}) & =& \varphi_{S_1,\beta_1} \circ \varphi_{S_2,\beta_2} (\{x_1,x_2\}) &=
    & \varphi_{S_1,\beta_1} (\{x_1,y_2\}) &=& \{y_1,y_2\};\\  [2pt]
     \psi_\beta (\{x_1,z\}) & =& \varphi_{S_1,\beta_1} \circ \varphi_{S_2,\beta_2} (\{x_1,z\}) &=
    & \varphi_{S_1,\beta_1} (\{x_1,z\}) &=& \{y_1,z\};\\ [2pt]
    \psi_\beta (\{x_1,y_2\}) & =& \varphi_{S_1,\beta_1} \circ \varphi_{S_2,\beta_2} (\{x_1,y_2\}) &=
    & \varphi_{S_1,\beta_1} (\{x_1,x_2\}) &=& \{y_1,x_2\};\\  [2pt]
     \psi_\beta (\{y_1,x_2\}) & =& \varphi_{S_1,\beta_1} \circ \varphi_{S_2,\beta_2} (\{y_1,x_2\}) &=
    & \varphi_{S_1,\beta_1} (\{y_1,y_2\}) &=& \{x_1,y_2\};\\ [2pt]
    \psi_\beta (\{y_1,z\}) & =& \varphi_{S_1,\beta_1} \circ \varphi_{S_2,\beta_2} (\{y_1,z\}) &=
    & \varphi_{S_1,\beta_1} (\{y_1,z\}) &=& \{x_1,z\};\\  [2pt]
     \psi_\beta (\{y_1,y_2\}) & =& \varphi_{S_1,\beta_1} \circ \varphi_{S_2,\beta_2} (\{y_1,y_2\}) &=
    & \varphi_{S_1,\beta_1} (\{y_1,x_2\}) &=& \{x_1,x_2\};\\ [2pt]
    \psi_\beta (\{w,x_2\}) & =& \varphi_{S_1,\beta_1} \circ \varphi_{S_2,\beta_2} (\{w,x_2\}) &=
    & \varphi_{S_1,\beta_1} (\{w,y_2\}) &=& \{w,y_2\};\\  [2pt]
     \psi_\beta (\{w,z\}) & =& \varphi_{S_1,\beta_1} \circ \varphi_{S_2,\beta_2} (\{w,z\}) &=
    & \varphi_{S_1,\beta_1} (\{w,z\}) &=& \{w,z\};\\[2pt]
    \psi_\beta (\{w,y_2\}) & =& \varphi_{S_1,\beta_1} \circ \varphi_{S_2,\beta_2} (\{w,y_2\}) &=
    & \varphi_{S_1,\beta_1} (\{w,x_2\}) &=& \{w,x_2\};\\ [2pt] 
     \psi_\beta (\{w,x_1\}) & =& \varphi_{S_1,\beta_1} \circ \varphi_{S_2,\beta_2} (\{w,x_1\}) &=
    & \varphi_{S_1,\beta_1} (\{w,x_1\}) &=& \{w,x_1\};\\ [2pt]
    \psi_\beta (\{x_1,y_1\}) & =& \varphi_{S_1,\beta_1} \circ \varphi_{S_2,\beta_2} (\{x_1,y_1\}) &=
    & \varphi_{S_1,\beta_1} (\{x_1,y_1\}) &=& \{x_1,y_1\};\\[2pt]
     \psi_\beta (\{w,y_1\}) & =& \varphi_{S_1,\beta_1} \circ \varphi_{S_2,\beta_2} (\{w,y_1\}) &=
    & \varphi_{S_1,\beta_1} (\{w,y_1\}) &=& \{w,y_1\};\\ [2pt]
    \psi_\beta (\{x_2,y_2\}) & =& \varphi_{S_1,\beta_1} \circ \varphi_{S_2,\beta_2} (\{x_2,x_2\}) &=
    & \varphi_{S_1,\beta_1} (\{x_2,y_2\}) &=& \{x_2,y_2\};\\  [2pt]
     \psi_\beta (\{x_2,z\}) & =& \varphi_{S_1,\beta_1} \circ \varphi_{S_2,\beta_2} (\{x_2,z\}) &=
    & \varphi_{S_1,\beta_1} (\{x_2,z\}) &=& \{x_2,z\};\\ [2pt]
    \psi_\beta (\{z,y_2\}) & =& \varphi_{S_1,\beta_1} \circ \varphi_{S_2,\beta_2} (\{z,y_2\}) &=
    & \varphi_{S_1,\beta_1} (\{z,y_2\}) &=& \{z,y_2\}.\\ 
\end{array}
\]\qed
 \end{example}
%
Corollary~\ref{cor:2cut} implies that for every $1 \le i < j \le |q(G)|$,  we have that 
\begin{equation}\label{eq:1}
 \varphi_{S_i,\beta_i} \circ \varphi_{S_j,\beta_j} = \varphi_{S_j,\beta_j} \circ \varphi_{S_i,\beta_i}.
 \end{equation}
We can assign to $\beta$ its characteristic vector indexed by the elements of $\mathcal{T}_G$.
 Thus, by (\ref{eq:2}) and (\ref{eq:1}), we have that
\[\mathcal{N}:=\left \langle \{\psi_\beta: \beta \in \mathcal{T}_G\} \right \rangle \cong \mathbb{Z}_2^{|\mathcal{T}_G|}.\]

\subsection*{$\aut(G)$ acting on $\mathcal{T}_G$}

Let $f \in \aut(G)$. Note that for every $1 \le i \le q(G)$, we have that $f(S_i) \in Q(G)$. 
Abusing notation let $f(i)$ be the index such that \[f(S_i)=S_{f(i)}.\]
 Let $r:=c_{S_i}=c_{S_{i_f}}$; let $G_1,\dots,G_r$ and
$G_1',\dots,G_r'$ be the connected components of $G \setminus S_i$ and $G \setminus S_{i_f}$, respectively, in  their respective orders.
For every $1 \le j \le r$, abusing notation let $f(i)(j)$ be the index such that
\[f(G_j)=G_{f(i)(j)}'.\]
Let $t:=(k_1,\dots,k_{r})_{S_i} \in \mathcal{T}_G$, and let 
\[ft=f(k_1,\dots,k_{r})_{S_i}:=(k_{f(i)(1)},\dots,k_{f(i)(r)})_{S_{f(i)}} \in \mathcal{T}_G.\]
Note that

\LabelQuote{$A\in V_{S_i,t}$ if and only if $f(A)\in V_{S_{f(i)},ft}$.}{$\ast$}

We have defined an action of $\aut(G)$ on $\mathcal{G}$, since for every $t \in \mathcal{T}_G$,
we have that:
\begin{itemize}
 \item $e t=t$, where $e$ is the identity on $\aut(G)$; and
 \item $f(gt)=(f\circ g)t$ for every $f,g \in \aut(G)$.
\end{itemize}
We extend this action to the subsets of $\mathcal{T}_G$, by defining
\[f\beta:=\{ft:t \in \beta\},\]
for every $\beta \subseteq T_G$.

Suppose that $k=n/2$. We can define an action of $\langle \mathfrak{c} \rangle$ on
$\mathcal{T}_G$, by defining
 \[\mathfrak{c}(k_1,\dots,k_{r})_{S_i}:=(|G_1|-k_1,\dots,|G_r|-k_{r})_{S_i}.\]
 Note that

\LabelQuote{$A\in V_{S_i,t}$ if and only if $\mathfrak{c}(A) \in V_{S_{i},\mathfrak{c}t}$.}{$\dagger$}
As before, we extend this action to the subsets of 
 $\mathcal{T}_G$, by defining
\[\mathfrak{c}\beta:=\{\mathfrak{c}t:t \in \beta\},\]
for every $\beta \subseteq T_G$.

Recall that we have labeled the vertices of every $S_i$, so that $S_i=\{x_i,y_i\}$.
Let 
\[\mathcal{S}:=\left \{ f \in \aut(G):f(x_i)=x_{f(i)}  \textrm{ for every } 1 \le i \le q(G) \right \}.\]
That is, the elements of $\mathcal{S}$ send $x$'s to $x$'s, and $y$'s to $y$'s.
Thus, if $f$ and $g$ are two elements of $\mathcal{S}$, then so
are $f^{-1}$ and $f\circ g$. Therefore $\mathcal{S}$ is subgroup
of $\aut(G)$.

\begin{lemma}\label{lem:psi_com}
 Let $f \in \mathcal{S}$ and $\beta \in \mathcal{T}_G$. Then
 \begin{itemize}
 \item[$(a)$] \[\iota(f) \circ \psi_\beta = \psi_{f\beta} \circ \iota(f); \textrm{ and}\]
 \item[$(b)$] if $k=n/2$, then 
  \[\mathfrak{c} \circ \psi_\beta = \psi_{\mathfrak{c}\beta} \circ \mathfrak{c}.\]
  \end{itemize}
\end{lemma}
\begin{proof}
First we show that for every $1 \le i \le q(G)$, we have that
 \begin{itemize}
 \item[$(a')$] \[\iota(f) \circ \varphi_{S_i,\beta_i} = \varphi_{S_{f(i)},f\beta_{f(i)}} \circ \iota(f); 
 \textrm{ and } \]
 \item[$(b')$]  if $k=n/2$, then 
  \[\mathfrak{c} \circ \varphi_{S_i,\beta_i} = \varphi_{S_i,\mathfrak{c}\beta_i} \circ \mathfrak{c}.\]
 \end{itemize}
 
Let $A \in V(F_k(G))$. If $|A \cap S_i|\neq 1$, then
\[\varphi_{S_i,\beta_i}(A)=\varphi_{S_{f(i)},f\beta_{f(i)}}(A)=A,\]
if in addition $k=n/2$, then
\[\varphi_{S_i,\beta_i}(A)=\varphi_{S_i,\mathfrak{c}\beta_i}(A)=A.\]
Thus, $(a')$ and $(b')$ hold in this case. 

Suppose that $|A \cap S_i|= 1$, and 
let $t \in T_{S_i}$ such that $A \in V_{S_i,t}$. By $(\ast)$, we have that:
\begin{itemize}
\item if $t \in \beta_i$, then
\[\iota(f) \circ \varphi_{S_i,\beta_i}(A) =\iota(f)(A^{S_i})=\iota(f)(A)^{S_{f(i)}}= \varphi_{S_{i_f},f\beta_{f(i)}} \circ \iota(f)(A);\]
and
\item if $t \notin \beta_i$, then
\[\iota(f) \circ \varphi_{S_i,\beta_i}(A) =\iota(f)(A)= \varphi_{S_{f(i)},f\beta_{f(i)}} \circ \iota(f)(A).\]
\end{itemize}
Thus, ($a'$) holds. 

Suppose that $k=n/2$. By $(\dagger)$ we have that:
\begin{itemize}
\item if $t \in \beta_i$, then
\[\mathfrak{c} \circ \varphi_{S_i,\beta_i}(A) =\mathfrak{c}(A^{S_i})= \mathfrak{c}(A)^{S_i}= \varphi_{S_{i},\mathfrak{c}\beta_{i}} \circ \mathfrak{c}(A);\]
and
\item if $t \notin \beta_i$ then
\[\mathfrak{c} \circ \varphi_{S_i,\beta_i}(A) =\mathfrak{c}(A)= \varphi_{S_{i},\mathfrak{c}\beta_{i}} \circ \mathfrak{c}(A).\]
\end{itemize}
Thus, ($b'$) holds.

($a'$) and (\ref{eq:1}) imply that 
\begin{align*}
 \iota(f) \circ \psi_\beta &= \iota(f) \circ \varphi_{S_1,\beta_1} \circ \varphi_{S_2\beta_2} \circ \cdots \circ \varphi_{S_{q(G),\beta_q(G)}} \\
 &=  \varphi_{S_{f(1)},f\beta_{f(1)}} \circ \varphi_{S_{f(2)},f\beta_{f(2)}} \circ \cdots \circ \varphi_{S_{q(G)_f,f\beta_{f(q(G))}}} \circ \iota(f) \\
  &=  \varphi_{S_{1},f\beta_{1}} \circ \varphi_{S_{2},f\beta_{2}} \circ \cdots \circ \varphi_{S_{q(G),f\beta_{q(G)}}} \circ \iota(f) \\
  &= \psi_{f\beta}\circ \iota(f).
\end{align*}
Thus, $(a)$ holds.

Suppose that $k=n/2$. ($b'$)  implies that 
\begin{align*}
 \mathfrak{c} \circ \psi_\beta &= \mathfrak{c} \circ \varphi_{S_1,\beta_1} \circ \varphi_{S_2\beta_2} \circ \cdots \circ \varphi_{S_{q(G)},\beta_q(G)} \\
 &=  \varphi_{S_{1},\mathfrak{c}\beta_{1}} \circ \varphi_{S_{2},\mathfrak{c}\beta_{2}} \circ \cdots \circ \varphi_{S_{q(G)},\mathfrak{c}\beta_{q(G)}} \circ \mathfrak{c} \\
  &= \psi_{\mathfrak{c}\beta}\circ \mathfrak{c}.
\end{align*}
Thus, $(b)$ holds.
\end{proof}

We are ready to show our main theorem.
\begin{theorem}\label{thm:main}
Let $G$ be a connected graph. Then $\aut(F_k(G))$ contains a subgroup
\[\mathcal{G}\cong   
\begin{cases}
\mathbb{Z}_2 \wr_{\mathcal{T}_G} \iota(\mathcal{S}) & \textrm{ if } k\neq n/2;\\
\mathbb{Z}_2 \wr_{\mathcal{T}_G} \left (\iota(\mathcal{S}) \times \mathbb{Z}_2 \right ) & \textrm{ if } k=n/2;
\end{cases}
\]
such that $\iota(\aut(G)) \le \mathcal{G}$.
\end{theorem}
\begin{proof}

Since for $k=n/2$, we have that $\mathfrak{c}\circ \iota(f) = \iota(f) \circ \mathfrak{c}$ for every $f \in \aut(G)$, we have that
\[\langle \iota(\mathcal{S}) \cup \{\mathfrak{c}\} \rangle=\iota(\mathcal{S})\times \mathbb{Z}_2,\]
in this case.
Let 
\[\mathcal{G}:= \begin{cases}
                 \iota(\mathcal{S})\mathcal{N} & \textrm{ if } k\neq n/2; \textrm{ and } \\
                 \left(\langle \iota(\mathcal{S}) \cup \{\mathfrak{c}\} \rangle \right )\mathcal{N} & \textrm{ if } k = n/2.                 
                \end{cases}
\]
Corollary~\ref{cor:psi} implies that
\[\{e\}=\begin{cases}
         \iota(\mathcal{S})\cap \mathcal{N} &\textrm{ if } k\neq n/2;  \\
         \langle \iota(\mathcal{S}) \cup \{\mathfrak{c} \}\rangle\cap \mathcal{N} &\textrm{ if } k\neq n/2.
        \end{cases}
\]
Lemma~\ref{lem:psi_com} implies that 
\begin{itemize}
 \item $\mathcal{G}$ is subgroup of $\aut(F_k(G))$; and 
 \item  $\mathcal{N}$ is normal subgroup of $\mathcal{G}$. 
\end{itemize}
Therefore,
\[\mathcal{G}\cong   
\begin{cases}
\mathbb{Z}_2 \wr_{\mathcal{T}_G} \iota(\mathcal{S}) & \textrm{ if } k\neq n/2;\\
\mathbb{Z}_2 \wr_{\mathcal{T}_G} \left (\iota(\mathcal{S}) \times \mathbb{Z}_2 \right ) & \textrm{ if } k=n/2.
\end{cases}
\]

Since \[\iota(S) \mathcal{N} \le \mathcal{G},\]
it is sufficient to show that
\[\iota(\aut(G)) \le \iota(S)\mathcal{N}.\] 
Let $f \in \aut(G)$; let \[I_f:=\{i:1 \le i \le q(G) \textrm{ and } f(x_i)=y_{f(i)}\} \textrm{ and } U_f:=\bigcup_{i \in I_f} S_i;\]
and let $f':V(G)\to V(G)$ be the map that sends every $v \in V(G)$ to
 \[f'(v)=\begin{cases} f(v) &\textrm{ if } v \notin U_f; \\ 
                    x_{f(i)} &\textrm{ if } v=x_i \textrm{ and } i \in I_f;\\
                    y_{f(i)} &\textrm{ if } v=y_i \textrm{ and } i \in I_f.
                    \end{cases}\]
Note that $f' \in \mathcal{S}$. 
Let 
\[\beta=\bigcup_{i \in I_f} T_{S_i}.\]
We have that
\[\iota(f)=\psi_{\beta} \circ \iota(f').\]
Therefore,
\[\iota( \aut(G) ) \le \iota(S)\mathcal{N}.\] 
\end{proof}

 \bibliographystyle{alpha} 
\bibliography{Automorphism}

 \end{document}